\documentclass[11pt]{amsart}
\usepackage{amscd}
\usepackage[all]{xy}
\usepackage{graphicx}
\usepackage{amsmath}
\usepackage{amsfonts}
\usepackage{amssymb}
\usepackage{latexsym}
\usepackage{slashed}
\usepackage{soul}
\usepackage{comment}
\usepackage{color}

\usepackage{amsmath, latexsym, amssymb}
\numberwithin{equation}{section}
\theoremstyle{plain}
\newtheorem{lemma}{Lemma}[section]

\newtheorem{theorem}[lemma]{Theorem}
\newtheorem{corollary}[lemma]{Corollary}

\theoremstyle{definition}
\newtheorem{definition}[lemma]{Definition}
\newtheorem{remark}[lemma]{Remark}
\newtheorem{example}[lemma]{Example}

\begin{document}
\newcommand{\R}{{\mathbb R}}
\newcommand{\C}{{\mathbb C}}
\newcommand{\F}{{\mathbb F}}
\renewcommand{\O}{{\mathbb O}}
\newcommand{\Z}{{\mathbb Z}} 
\newcommand{\N}{{\mathbb N}}
\newcommand{\Q}{{\mathbb Q}}
\newcommand{\Sb}{{\mathbb S}}
\renewcommand{\H}{{\mathbb H}}
\renewcommand{\P}{{\mathbb P}}
\newcommand{\Aa}{{\mathcal A}}
\newcommand{\Bb}{{\mathcal B}}
\newcommand{\Cc}{{\mathcal C}}    
\newcommand{\Dd}{{\mathcal D}}
\newcommand{\Ee}{{\mathcal E}}
\newcommand{\Ff}{{\mathcal F}}
\newcommand{\Gg}{{\mathcal G}}    
\newcommand{\Hh}{{\mathcal H}}
\newcommand{\Kk}{{\mathcal K}}
\newcommand{\Ii}{{\mathcal I}}
\newcommand{\Jj}{{\mathcal J}}
\newcommand{\Ll}{{\mathcal L}}    
\newcommand{\Mm}{{\mathcal M}}    
\newcommand{\Nn}{{\mathcal N}}
\newcommand{\Oo}{{\mathcal O}}
\newcommand{\Pp}{{\mathcal P}}
\newcommand{\Qq}{{\mathcal Q}}
\newcommand{\Rr}{{\mathcal R}}
\newcommand{\Ss}{{\mathcal S}}
\newcommand{\Tt}{{\mathcal T}}
\newcommand{\Uu}{{\mathcal U}}
\newcommand{\Vv}{{\mathcal V}}
\newcommand{\Ww}{{\mathcal W}}
\newcommand{\Xx}{{\mathcal X}}
\newcommand{\Yy}{{\mathcal Y}}
\newcommand{\Zz}{{\mathcal Z}}

\newcommand{\Ds}{{\slashed D}}

\newcommand{\zt}{{\tilde z}}
\newcommand{\xt}{{\tilde x}}
\newcommand{\Ht}{\widetilde{H}}
\newcommand{\ut}{{\tilde u}}
\newcommand{\Mt}{{\widetilde M}}
\newcommand{\Llt}{{\widetilde{\mathcal L}}}
\newcommand{\yt}{{\tilde y}}
\newcommand{\vt}{{\tilde v}}
\newcommand{\Ppt}{{\widetilde{\mathcal P}}}
\newcommand{\bp }{{\bar \partial}} 
\newcommand{\ad}{{\rm ad}}
\newcommand{\Om}{{\Omega}}
\newcommand{\om}{{\omega}}
\newcommand{\eps}{{\varepsilon}}
\newcommand{\Di}{{\rm Diff}}
\renewcommand{\Im}{{ \rm Im \,}}
\renewcommand{\Re}{{\rm Re \,}}

\renewcommand{\a}{{\mathfrak a}}
\renewcommand{\b}{{\mathfrak b}}
\newcommand{\e}{{\mathfrak e}}
\renewcommand{\k}{{\mathfrak k}}
\newcommand{\m}{{\mathfrak m}}
\newcommand{\pg}{{\mathfrak p}}
\newcommand{\g}{{\mathfrak g}}
\newcommand{\gl}{{\mathfrak gl}}
\newcommand{\h}{{\mathfrak h}}
\renewcommand{\l}{{\mathfrak l}}
\newcommand{\sm}{{\mathfrak m}}
\newcommand{\n}{{\mathfrak n}}
\newcommand{\s}{{\mathfrak s}}
\renewcommand{\o}{{\mathfrak o}}
\renewcommand{\so}{{\mathfrak so}}
\renewcommand{\u}{{\mathfrak u}}
\newcommand{\su}{{\mathfrak su}}
\newcommand{\X}{{\mathfrak X}}

\newcommand{\ssl}{{\mathfrak sl}}
\newcommand{\ssp}{{\mathfrak sp}}
\renewcommand{\t}{{\mathfrak t }}
\newcommand{\Cinf}{C^{\infty}}
\newcommand{\la}{\langle}
\newcommand{\ra}{\rangle}
\newcommand{\ha}{\scriptstyle\frac{1}{2}}
\newcommand{\p}{{\partial}}
\newcommand{\notsub}{\not\subset}
\newcommand{\iI}{{I}}               
\newcommand{\bI}{{\partial I}}      
\newcommand{\LRA}{\Longrightarrow}
\newcommand{\LLA}{\Longleftarrow}
\newcommand{\lra}{\longrightarrow}
\newcommand{\LLR}{\Longleftrightarrow}
\newcommand{\lla}{\longleftarrow}
\newcommand{\INTO}{\hookrightarrow}

\newcommand{\QED}{\hfill$\Box$\medskip}
\newcommand{\UuU}{\Upsilon _{\delta}(H_0) \times \Uu _{\delta} (J_0)}
\newcommand{\bm}{\boldmath}
\newcommand{\coker}{\mbox{coker}}

\def\hook{\mbox{}\begin{picture}(10,10)\put(1,0){\line(1,0){7}}
  \put(8,0){\line(0,1){7}}\end{picture}\mbox{}}

\title[McLean's  second variation formula]{\large McLean's  second variation formula  revisited}
\author[H.V. L\^e  and  J. Van\v zura]{H\^ong V\^an L\^e and Ji\v r\'i  Van\v zura}
 \date{\today}
 \thanks {HVL and JV are partially supported by RVO: 67985840}
\date{\today}
\medskip
\address{Institute  of Mathematics,   CAS,
Zitna 25, 11567  Praha 1, Czech Republic} 

\abstract 
We   revisit       McLean's  second variation formulas for  calibrated submanifolds  in exceptional geometries, and correct  his formulas  concerning  associative  submanifolds  and Cayley  submanifolds, using  a  unified treatment   based  on the  (relative)   calibration method and Harvey-Lawson's identities.   
\endabstract 
\maketitle
\tableofcontents


\section{Introduction}\label{sec:intro} 
Calibrated  geometry has been invented by Harvey-Lawson in 1982 \cite{HL1982} motivated  by rich  theories of  complex  manifolds, exceptional geometries  and  minimal  submanifolds. We  refer the  reader   to \cite{Morgan2009} for  an extensive   survey   on calibration  method. 
In 1998 McLean  published a paper on deformation of calibrated submanifolds  \cite{McLean1998}, inspired  by  similarities  between  calibrated submanifolds and  complex  submanifolds.
One important part  of his  study  is  the   second variation  of volume of compact calibrated  submanifolds, which is  also the subject of  our note.  
McLean   distinguished  two  families of calibrated  submanifolds  in  exceptional geometries. The first  family consists  of  special  Lagrangian and  coassociative  submanifolds.
The second family consists of  associative and Cayley submanifolds. In the first family   the normal bundle of a calibrated  submanifold is  isomorphic  to a vector bundle  intrinsic  to the submanifold, namely  the normal bundle of a   special Lagrangian  submanifold $L$ is isomorphic  to the tangent bundle $TL$
(or the cotangent bundle  $T^*L$ via  the metric)  and the normal bundle  of a  coassociative  submanifold $L$  is isomorphic to  the  bundle  of self-dual two-forms.  From a computational point of view, special Lagrangian  and  coassociative   submanifolds $L$ can be defined in terms of vanishing of closed forms on $L$. 
Moreover    deformation of calibrated submanifolds   in this family  is unobstructed.
  In the  second  family the normal bundle of a calibrated  submanifold is not intrinsic, namely  the normal  bundle of an associative  submanifold  $L$ is trivial 
  (Lemma \ref{lem:lemcorti}) and the normal  bundle of a Cayley submanifold is  a twisted spinor  bundle \cite[Section 6]{McLean1998}.
From  a computational  point of view,  associative  and Cayley  submanifolds  cannot be defined  in terms  of the vanishing of closed  forms, but they can be defined in terms  of the vanishing
of  certain vector valued  forms. 
 In particular, deformation theory   for  calibrated submanifolds  in the second family has a different character than  the  one for the first family. 

In \cite[Theorem 2.4, p. 711]{McLean1998}, using     moving frame  method, McLean derived  a general   formula for  the second variation of the  volume of a compact calibrated  submanifold.
Applying this formula to calibrated submanifolds of the first and second family  he obtained   formulas    which   are similar to Simons'
  second variation formula for  K\"ahler
submanifolds \cite[p. 78]{Simons1968}.  McLean's second variation  formula  for special Lagrangian submanifolds  has been  revisited by  L\^e-Schwachh\"ofer in \cite{LS2014}, where
they  extended  the  relative  calibration method  developed by  L\^e in \cite{Le1989, Le1990} to derive  the second variation  formula for compact Lagrangian submanifolds in strict nearly K\"ahler 6-manifolds.  They  also indicated  how their method  can be applied to   calibrated submanifolds, whose corresponding  calibration  satisfies  the long  version  of    Harvey-Lawson's  identity, see Remark  \ref{rem:HLI}. As an example, they analyzed  the second variation formula for  special  Lagrangian submanifolds.

In this  note we   revisit McLean's second variation  formula  for  associative  and   coassociative  submanifolds  in  $G_2$-manifolds, and Cayley submanifolds  in   Spin(7)-manifolds.  Our  main observation is     that   all  calibrations  under consideration   satisfy the following {\it Harvey-Lawson's  identity}.

\begin{definition}\label{def:HLI}  A   calibration  $\varphi\in  \Lambda ^k (\R^n)^*$    is  said to satisfy  {\it Harvey-Lawson's  identity},  if there
exists a vector   valued $k$-form $\Psi \in  \Lambda ^k (\R^n)^* \otimes  \R^m$ such that
\begin{equation}
[\varphi (\xi)]^2  + | \Psi (\xi)| ^2  = |\xi|  ^2  \text{  for  all } \xi \in Gr_k (\R^n).\label{eq:HLI1}
\end{equation}
Let $M$ be  a Riemannian   manifold. A calibration $\varphi \in \Om  ^k(M^n)$  is said to satisfy {\it Harvey-Lawson's identity}, if there exists
a Riemannian vector bundle  $E$ over  $M^n$  and   an $E$-valued  $k$-form $\Psi \in \Om^k(M, E)$ such that  for all $ x\in M^n$ we have
\begin{equation}
[\varphi (\xi)]^2  + | \Psi (\xi)| ^2  = |\xi|  ^2  \text{  for  all } \xi \in Gr_k (T_xM^n).\label{eq:HLI2}
\end{equation}
\end{definition}

\begin{remark}\label{rem:HLI}   Harvey-Lawson's   identity appears  many times   in  \cite{HL1982}, but instead  of $| \Psi (\xi)|^2$
Harvey-Lawson usually wrote {\it its long version}  $\sum_i |\Psi _i (\xi)| ^2 $, 
see Question 6.5  and Formula (6.6) in \cite[p. 68]{HL1982}, as   well as  Formulas (6.16) (p. 71),  (6.17) (p.73), Theorem  1.7 (p.  88) of the cited paper   and  Lemmas \ref{lem:hlass}, \ref{lem:tau7}, \ref{lem:hlcay} below.   All  calibrated  submanifolds  considered   in McLean's  paper have corresponding calibrations  that satisfy  Harvey-Lawson's  identity, see also  Remark \ref{rem:gen} below.
\end{remark}

In this    note  we  prove the  following.

\begin{theorem}[Main Theorem]\label{thm:main}    Let $\varphi$ be a calibration on    a Riemannian manifold  $M$  and   $\Psi \in \Om^* (M, E)$  such that
$\varphi$ and $\Psi$ satisfy  Harvey-Lawson's  identity  (\ref{eq:HLI2}). Assume that  $L$ is a compact oriented $\varphi$-calibrated  submanifold  and $V$ is
normal  vector  field  on  $L$.  Then  the  second  variation  of the volume  of $L$  with variation  field  $V$  is given  by
$${d^2\over dt^2}|_{ t =0} vol(L_t) = \int _L | \nabla _{\p t}|_{ t =0} \Psi (( \exp  t V) _* ( \xi(x)))|^2  dvol_x.$$
Here\\
$\bullet$ $\xi (x)$ is  the unit decomposable $k$-vector  that is associated to $T_x L$,\\ 
$\bullet$ $\exp tV$ denotes the flow  on  a neighborhood  of  $L$ that is generated  by   a  vector  field  whose value at $x \in L$  is equal to $V$,
and $L_t = \exp tV (L)$,\\
$\bullet$ $ \Psi (( \exp  t V) _* ( \xi(x)))\in E_{ \exp  t V  (x)}  $  and  $\nabla _{\p t}\Psi (( \exp  t V) _* ( \xi(x)))$ denotes  the   covariant  derivative   of the  section  $\Psi (( \exp  t V) _* ( \xi_x))\in E|_{ [\exp  t V  (x) ]}$ of the restriction of the vector  bundle $E$ to  the  curve  $[\exp   tV (x)] \subset  M$,  \\
$\bullet$  we assume that   the covariant derivative $ \nabla_{ \p t}$ of $E$ along the curve  $[\exp  tV (x)]$   preserves  the metric  on $E$.
\end{theorem}
Note that the equation  $\nabla_{\p t}|_{t =0}  \Psi((\exp  tV)_*(\xi(x))) = 0$  describes  the equation  for  $V$  to be an infinitesimal  deformation of  $\varphi$-calibrated  submanifold  $L$.
Thus  the second  variation formula  for  calibrated  submanifolds in Theorem \ref{thm:main} follows  from  the equation  for  Zariski tangent  vectors of the moduli space  of  calibrated submanifolds under consideration.  

From    our Main Theorem we obtain immediately  the following.

\begin{corollary}\label{cor:jac} Let $L$  be a   $\varphi$-calibrated submanifold  in  Theorem \ref{thm:main}. Then a normal vector  field $V$  on $L$ is an  infinitesimal deformation of $\varphi$-calibrated submanifolds if and only if  $V$ is a Jacobi vector  field  on $L$,  regarding $L$ as a  minimal  submanifold.
\end{corollary}

In fact,    Corollary  \ref{cor:jac}   holds  for any    compact  calibrated   submanifold  without  validity of  Harvey-Lawson's  identity, see Remark \ref{rem:gen} below.

As  applications of the Main Theorem,  we  shall derive simple  second variation formulas  for  associative,  coassociative  and Cayley submanifolds respectively, which agree   with McLean's  formulas up to a multiplicative  constant.

Our note is organized  as follows. In    Section \ref{sec:main}  we give  a  proof  of Theorem  \ref{thm:main}  and   discuss  a slight   generalization of it in Remark \ref{rem:gen}.
In  Section \ref{sec:g2} we  give   a new   proof  of McLean's    second variation formula   for  associative, coassociative  submanifolds  in $G_2$-manifolds
and derive from it a second  variation formula for  special Lagrangian   submanifolds  in Calabi-Yau 6-manifolds.  In  Section  \ref{sec:spin7}  we give a  new proof of McLean's   second  variation formula  for Cayley submanifolds. (Our    formulas  for associative  and Cayley  submanifolds  differ  from  McLean's  formulas by  a scaling factor). 
At the end of our note we 
explain  where   McLean did  mistakes in his computations (Remark \ref{rem:mistake}).

\section{Proof of the Main Theorem}\label{sec:main}
\begin{proof}[Proof of Theorem \ref{thm:main}]  Let us keep notations  in the previous section, in particular,  in  Theorem \ref{thm:main}.  Abusing the  notation, denote by $V$    a  vector  field  in a neighborhood of $L$    whose  value  at $L$ is the given  normal vector field  $V$, see     explanation in Theorem 
\ref{thm:main}.
Set 
$$\xi_t (x): = (\exp tV)_* (\xi(x)),$$
and
 $$g_t|_{L} : = (\exp  tV)^*  g|_{\exp  tV (L)},$$  
where $g|_{\exp  tV (L)}$ denotes the   metric on  $\exp  tV (L)$  induced  from the  ambient metric  on $M$.
  Denote by  $vol_t $   the    induced volume form on $L$ associated to $g_t$.
  Since  $  vol_t (x) = \det (g_{ij})  ^{1/2}  dx  =  |\xi_t  (x)| \cdot vol_0 (x)$,   taking into  account the minimality of $L$,
we observe  that 
for   all  $x \in L$
\begin{equation}
|\xi_0(x)| =1 \text {  and } \frac{d}{dt}|_{t =0}   |(\xi_t (x))|  = 0 .\label{eq:min}
\end{equation}

Hence

\begin{eqnarray}
\frac{d^2}{dt^2} |_{t =0} vol (\exp tV (L)) =  \int _{L}\frac{d^2}{dt^2} |_{t =0}|(\xi_t (x)|\, d\, vol_x\nonumber \\
 = \frac{1}{ 2} \int _{L}\frac{d^2}{dt^2} |_{t =0}|\xi_t (x)|^2\, d\, vol_x.\label{eq:jac1}
\end{eqnarray}

To simplify notation, we    write
$$ D (V) (x): = \nabla _{\p t}|_{ t =0} \Psi (\xi_t(x)).$$

 
\begin{lemma}\label{lem:chi}  For all $x\in L$ we have
$$\frac{d^2}{dt^2}|_{t=0} |\Psi(\xi_t(x))|^2 = 2|D(V)|^2 (x) .$$
\end{lemma}
\begin{proof}  
We compute
$$\frac{d^2}{dt^2}|_{t=0} |\Psi(\xi_t(x))|^2 =    2 \frac{d}{dt}_{|t = 0} \la \nabla _{\p  t}  (\Psi (\xi_t )) , \Psi(\xi_t) \ra (x) $$
$$ = 2|D(V)|^2(x)+  2\la\Psi(\xi(x)) , \nabla_{\p t} \nabla _{\p t} (\Psi(\xi_t ))(x)|_{t=0}\ra . $$
Since $\Psi(\xi(x)) = 0$ by Harvey-Lawson's  identity, this completes  the proof of  Lemma \ref{lem:chi}.
\end{proof}

\begin{lemma}\label{lem:varphi} We  have
$$\frac{d^2}{dt^2}|_{t =0}\int_{L}\varphi(\xi_t) ^2 \, dvol_x = 0.$$ 
\end{lemma}
\begin{proof}  
Since $L$ is   a $\varphi$-calibrated submanifold,  by \cite[Proposition 2.2. (ii)]{Le1989},  see also  \cite[Proposition 1.2  (ii)]{Le1990},  we have
\begin{equation}
(V \rfloor   \varphi)_{| L } = 0.\label{eq:cousin}
\end{equation}
Using $d\varphi = 0$, we obtain   from (\ref{eq:cousin})
\begin{equation}
(\Ll_V \varphi)|_L (x) = 0 \text  { for all }  x\in L.\label{eq:Lie1}
\end{equation}
Now  let us compute
\begin{equation}
\frac{d^2}{dt^2}|_{t =0}\int_{L}\varphi(\xi_t) ^2 \, dvol_x = 2\frac{d}{dt}_{|  t =0} \int_{L}\varphi(\xi_t) \cdot \frac{ d} {dt}  ( \exp t V) ^* \varphi (\xi(x)) \, dvol_x .\label{eq:Lie2}
\end{equation}
Using (\ref{eq:Lie1}), we obtain from  (\ref{eq:Lie2}), noting that $\varphi (\xi(x)) = 1$
$$\frac{d^2}{dt^2}|_{t =0}\int_{L}\varphi(\xi_t) ^2 \, dvol_x = 2\int _{L} \Ll _V   d ( V \rfloor \varphi ) = 0.$$
This completes  the  proof  of  Lemma \ref{lem:varphi}.
\end{proof}

Now let us complete  the proof  of Theorem \ref{thm:main}.      
Using  (\ref{eq:jac1}), Harvey-Lawson's identity (\ref{eq:HLI2})  and Lemmas \ref{lem:chi}, \ref{lem:varphi}, we obtain
\begin{eqnarray}
\frac{d^2}{dt^2} |_{t =0} vol (\exp t V (L))= \frac{1}{2} \frac{d^2}{dt^2}|_{t=0} \int_L \la \Psi (\xi_t(x)), \Psi (\xi_t (x))\ra dvol_x\nonumber \\
= \int_L | D (V)|^2dvol_x.\label{eq:jac2}
\end{eqnarray}
This completes the proof of Theorem  \ref{thm:main}.
\end{proof}

\begin{remark}\label{rem:gen}   Any  calibration $\varphi$ on a  Riemannian  manifold  $M$  satisfies  a  weak version  of Harvey-Lawson's identity (\ref{eq:HLI2}), where  we replace  $\Phi \in \Om^*(M, E)$  by a real function,  also denoted by $\Phi$,  on the  Grassmannian  of    oriented  $k$-decomposable  vectors  in   $TM$. In this case, using the argument  of the proof of Theorem \ref{thm:main},  the  function under integral  in the  RHS of the formula  in Theorem  \ref{thm:main} is replaced by
$(\p _t |_{ t =0} \Phi (\xi _t  (x)))^2$.  Thus  Corrollary  \ref{cor:jac}  also holds  for   any calibrated  submanifold.
\end{remark}

\section{Second variation formula  for associative  and coassociative submanifolds}\label{sec:g2}

\subsection{Associative  and coassociative   submanifolds}\label{subs:ass}
In this subsection we recall basic  definitions    of  associative 3-submanifolds   and coassociative  4-submanifolds in  a $G_2$-manifold  $(M^7, \varphi, g)$   and show that   the associated calibrations  satisfy Harvey-Lawson's  identity (Lemmas \ref{lem:hlass}, \ref{lem:tau7}).

Let $\O$ denote the octonion algebra. 
Denote by $\la,\ra$  the scalar product on $\O$  and   by $ \cdot $   the  octonion  multiplication. 
Recall that the  {\it associative   3-form}  $\varphi$ on $\Im \O$  is defined  as follows \cite[(1.1), IV.1.A, p. 113]{HL1982}
\[
\varphi(x, y, z) : =  \la  x\cdot y , z \ra = \la  x, y\cdot z \ra.
\]
Let $\Im \O = \R^7$  have coordinates $(x^1, \cdots, x^7)$.    
We abbreviate   $dx^i \wedge  dx^j \wedge dx^k$  as $x ^{ijk}$.  
 The fundamental {\it associative} 3-form  $\varphi$
	can be written in     coordinate  expression as follows \cite[(1.2), p. 113]{HL1982}
\begin{equation} \label{eq:phi}
 \varphi = x^{123}+x^{145} -x^{167} + ^{246}+x^{257} +x^{347} - x^{356}.
 \end{equation}
Its dual
\begin{equation} \label{eq:starphi}
*\varphi=  x^{4567} + x^{2367} - x^{2345} + x^{1357} +x^{1346} + x^{1256} -x^{1247}
\end{equation}
is called  the {\it  coassociative     form}.

 It is  well-known that $G_2$, the  automorphism  group of $\O$, is  also the  subgroup of  $GL(\R^7)$  that preserves $\varphi$ (resp. $*\varphi$).
Let $g_0$ denote  the standard  Euclidean  metric  on $\R^7$.
We call $(\varphi_0, g_0)$   {\it  the standard  $G_2$-structure}.

Let $M^7$ be an oriented 7-manifold and 
$\varphi$  a 3-form on $M^7$. 
A 3-form $\varphi$ is called a {\it $G_2$-structure} on $M^7$ if 
for each $p \in M^7$, there exists an oriented  linear isomorphism $I_p$ 
between $T_{p}M^7$ and $\mathbb{R}^{7}$ 
identifying $\varphi_{p}$ with $\varphi_{0}$. 
Then $\varphi$ induces the metric $g_\varphi$ 
by pulling back  $g_0$    via  $I_p$.  Since  $G_2$ is a subgroup of ${\rm SO}(7)$  the    metric $g_\varphi$ 
does not depend  on the choice of  $I_p$.

In our paper we are       concerned  only  with    $G_2$-manifolds $(M^7, \varphi, g)$, i.e.   the  $G_2$-structure  on $(M^7, \varphi, g)$  is torsion-free,   equivalently  $d\varphi = 0$ and $d * \varphi  = 0$.

A 3-submanifold  $L\subset   M^7$ is called {\it associative}, if  $\varphi|_{L} = vol_{L}$.
A  4-submanifold $L\subset M^7$  is called {\it  coassociative},  if $*\varphi|_{L} = vol_{L}.$

We  shall show that  $\varphi$ and $* \varphi$  satisfy  Harvey-Lawson's  identity.  
We set  (\cite[p. 114]{HL1982}
\cite[Definition IV.1.11,   Proposition IV.1.14, p. 116]{HL1982})
\begin{equation}
\la  \chi (x, y, z),  w\ra   : =  * \varphi (x, y,z, w).\label{eq:coass}
\end{equation} 

We regard $\chi$ as   an element in $\Om ^3 (M^7, TM^7)$.

The following  Lemma is  a Harvey-Lawson's identity.
\begin{lemma}\label{lem:hlass}(\cite[Theorem IV.1.6, p. 114]{HL1982}) For all  $x, y, z\in TM^7$ we have
$$\varphi(x, y, z) ^2 + |\chi (x, y,z) |^2 =  |x\wedge y \wedge   z|^2 .$$
\end{lemma}

 Now we  set   for  $x, y, z, w \in TM^7$ (\cite[(1.17),  Theorem 1.18,  p. 117]{HL1982})
\begin{equation}
 \tau(x, y, z, w): = -(\varphi(y, z, w) x+ \varphi(z, x, w) y+ \varphi(x, y, w) z+ \varphi(y,x,z) w).\label{eq:tau7}
\end{equation}

The following  Lemma is also  a Harvey-Lawson's identity

\begin{lemma}\label{lem:tau7} (\cite[Theorem IV.1.18, p. 117]{HL1982}) For all $x, y, z, w \in TM^7$ we  have
$$* \varphi(x, y, z, w) ^2 +|\tau(x, y, z, w)| ^2  =  |  x\wedge y \wedge z \wedge w | ^2 .$$
\end{lemma}

We  regard $\tau$  as an element in $\Om ^4(M^7, TM^7)$, see  also Remark \ref{rem:tau}.

\begin{example}\label{ex:lag} (\cite[12.2.1, p. 260]{Joyce2007}, \cite[p. 43]{CHNP2012}) Let  $(M^6, \om,\Om) $ be a Calabi-Yau manifold  with a fundamental 2-form $\om$ and a  complex  volume form $\Om$, see e.g. \cite{CS2002} for characterization of $SU(3)$-manifolds  via  $(\om, \Om)$. Denote by $g$ the associated Calabi-Yau metric on $(M^6, \om, \Om)$
Then $(S^1 \times M^6, d\theta \wedge \om + Re\,\Om,   d\theta ^2 + g)$  is a $G_2$-manifold.   If  $L$ is a  special  Lagrangian  submanifold in $(M^6, \om, \Om)$, then
$L_\theta: = \{ \theta\} \times L$  is  an associative  submanifold  in $ S^1 \times M^6$ for any $\theta \in  S^1$.
If $C$ is a complex curve in $M^6$, then  $S^1\times C$ is an associative
submanifold in $S^1 \times  M^6$.
\end{example}

\begin{example}\label{ex:lag1} (\cite[12.2.1, p. 260]{Joyce2007}) Let  $(M^6, \om,\Om) $ be a Calabi-Yau manifold  as above.
 If  $L$ is a  special  Lagrangian  submanifold in $(M^6, \om, \Om)$, then
$ S^1 \times L$  is  a coassociative  submanifold  in $( S^1 \times M^6, d\theta \wedge \om + Re\,\Om,   d\theta ^2 + g)$.  If $C$ is a complex surface in $M^6$, then  $C_\theta: =\{ \theta\} \times C$ is a  coassociative
submanifold in $S^1 \times  M^6$ for any $\theta \in  S^1$.
\end{example}

We refer the reader to \cite{Lotay2012, Kawai2014a, Kawai2014b}  for  consideration of homogeneous  associative  submanifolds in nearly  $G_2$-manifolds.

\begin{remark}\label{rem:associator}    An associative   3-form $\varphi$  defining   a $G_2$-structure  on a 7-manifold  $M^7$     can be   expressed   in terms of  the {\it cross  product}: $TM^ 7 \times TM^ 7 \to  TM^7$   defined  as follows
\cite[Definition B. 1, Appendix IV.B, p.  145]{HL1982}  
$$\varphi(x, y, z) = \la   x \times y,  z \ra. $$
\end{remark}

\subsection{The normal bundle  of  an    associative submanifold  and its associated Dirac  operator}\label{subs:spin}
We recall   known facts necessary  for    understanding    Formula (\ref{eq:infvar})   that enters  in  the proof of Theorem \ref{thm:jac}.
Our exposition  follows \cite[\S 5, p. 38-40]{CHNP2012},  and \cite[(1)-(5)]{Gayet2010}, see also Remarks \ref{rem:infvarm}, \ref{rem:mistake}  for comparison with McLean's formula. 

Let  $L$ be an   associative  3-fold  in a $G_2$-manifold  $(M^7, \varphi, g)$.   
 Since  $L$  is  orientable, it is  parallelizable,\footnote{the assertion is   well-known  for compact  orientable 3-manifolds. For  the proof of the case of   non-compact   orientable 3-manifolds  we refer the interested reader to    http://math.stackexchange.com/questions/1107682/elementary-proof-of-the-fact-that-any-orientable-3-manifold-is-parallelizable}   so we identify  $TL$ with  $L \times \Im \H$. Since  ${\rm rank\, }  NL > \dim L$,    therefore   there is a non-trivial section of $NL$.
  Using the  cross  product  $TL \times  NL \to NL$, we obtain the following

\begin{lemma}\label{lem:lemcorti} (\cite[Lemma 5.1, \S 5, p. 38]{CHNP2012}). The normal bundle  $NL$ of an associative submanifold $L$ is differentiably trivial.
\end{lemma}

Let $\nabla$ denote the Levi-Civita connection defined by the metric $g$ on $M^7$. 
Denote by $\nabla  ^\perp$ the  induced  connection  in the normal bundle  $NL$.

Using this, we  express  the Dirac operator $\Ds:  \Gamma (NL) \to \Gamma (NL)$ as follows.
For any $x \in L$  let $e_1, e_2, e_3$  denote a positive  orthonormal basis of $T_xL$  and for $V \in \Gamma (NL)$ we set
$$\Ds (V)_x : = \sum_{i =1} ^3  e_i \times (\nabla ^\perp _{e_i}  V) .$$

\begin{example}\label{ex:gayet}(\cite[Proposition 4.7]{Gayet2010}, cf. \cite[p. 43]{CHNP2012})  Let  $L$ be  a   special Lagrangian submanifold  in
a  Calabi-Yau  manifold $(M^6, \om, \Om)$.  Using notations in Example \ref{ex:lag}, for $\theta \in S^1$ we have   
$$N L_\theta = \R \oplus  NL, $$
where $NL$ is the normal  bundle of $L$  in   $M^6$. 
Then we identify
$$\Gamma (NL_\theta) \ni V =  f_V \oplus  \alpha _V\in \Om ^0 (L) \oplus \Om ^1(L)$$
where  $f_V \in \Om^0(L)$    and   $\alpha_V\in \Om^1 (L) $ is dual    to $JV\in \Gamma  (NL)$  w.r.t.  the      associated Riemannian  metric,  equivalently  $\alpha_V = V \rfloor \om$, see  \cite[Theorem 3.13, p. 723]{McLean1998}.  Using this identification we rewrite  the   Dirac operator  $\Ds:  \Gamma (NL_\theta) \to  \Gamma (NL_\theta)$   as  follows
$$ \Ds : \Om ^0 (L) \times \Om^1 (L) \to \Om^0(L)  \times \Om ^1 (L),$$
\begin{equation}
\Ds (f_V, \alpha_V) = (*d*\alpha_V,  -df_V -d*\alpha_V). \label{eq:sldirac}
\end{equation}
(The formula  in (\ref{eq:sldirac}) is identical with the formula in \cite[Proposition 4.7]{Gayet2010} and differs  from  the one in \cite[p.43]{CHNP2012}  by  the sign (-1), noting that   $d^* \alpha_V = - * d * \alpha _V$.)
\end{example} 

\subsection{Second   variation of  the  volume  of  an associative submanifold}

In this subsection  we give a new  proof  of  McLean's  second variation formula for associative submanifolds (Theorem \ref{thm:jac}), correcting  a  coefficient  in RHS of Formula (5.7) in \cite[p. 737]{McLean1998}),  which is twice larger than  our coefficient. Then we derive    from  Theorem \ref{thm:jac}   the McLean second  variation  formula for 
special Lagrangian submanifolds  in Calabi-Yau 6-manifolds (Example \ref{ex:2var}).

  We assume  that $L$ is a closed associative    submanifold  in a $G_2$-manifold  $M$.   
To compute  the second  variation of the  volume of $L$, by Theorem \ref{thm:main}  and Lemma \ref{lem:hlass}, it suffices  to  have the following.   


\begin{lemma}\label{lem:gayet}(\cite[(1)-(5)]{Gayet2010})  Let $\xi(x)$  denote the unit  decomposable  3-vector  associated with the tangent  space $T_xL$. Then  for any $V \in NL$ we have
\begin{equation}
\nabla_{\p t} |_{ t=0} ( \chi (\exp (tV)_*(\xi(x))) = \Ds (V)(x)  \in N_xL. \label{eq:infvar}
\end{equation}
\end{lemma}

\begin{theorem}\label{thm:jac}(cf. \cite[Theorem 5.3]{McLean1998}) Let $L$ be an associative  submanifold in   a $G_2$-manifold  $(M^7, \varphi, g)$.  For any   normal vector field $V$ on $L$ with compact support, the second  variation of  the volume  of $L$  with  the variation  field $V$  is given by
\begin{equation}
\frac{d^2}{dt^2} |_{t =0} vol  (L_t) = \int _{L} \la  \Ds (V), \Ds (V) \ra dvol_x.\label{eq:2var}
\end{equation}
\end{theorem}

\begin{proof} 
Clearly Theorem \ref{thm:jac} follows  from  Theorem \ref{thm:main}  and  Lemmas \ref{lem:hlass}, \ref{lem:gayet}.
\end{proof}

\begin{example}\label{ex:2var} We  shall derive  a formula  for  the  second variation of  the volume of   a special Lagrangian submanifold  $L$  in a Calabi-Yau
manifold $(M^6, \om, \Om)$ from  Theorem \ref{thm:jac}, using  notations and formulas in  Example \ref{ex:gayet}. Let $V$ be a normal   vector field on $L$ in $(M^6, \om, \Om)$.  Then $V$ is  also  a normal  vector field of the associative  submanifold $L_0\subset S^1 \times M^6$.  Let $\Phi_t$ denote the variation   associated  to $V$ in $M^6$.
Then $\tilde \Phi_t  : =Id \times  \Phi_t$ is the associated variation   of $L_0 \subset S^1 \times  M^6$.  Since $\tilde \Phi_t (L_0)$ is isometric to $\Phi_t (L)$, we have
$${d ^2 \over dt^2}_{|t =0} vol (\Phi _t  (L)) = {d^2 \over dt^2 }_{|t =0}  vol (\tilde  \Phi_t (L_0)).$$
Applying  Theorem \ref{thm:jac}, taking into  account (\ref{eq:sldirac}), we obtain
\begin{equation}
{d ^2 \over dt^2}_{|t =0} vol (\Phi _t  (L)) = \int_{L} (|*d* \alpha_V - *d \alpha_V| ^2)dvol_x $$
$$ = \int _{L} (| d^* \alpha_V| ^2  + | d\alpha _V| ^2)dvol_x .\label{eq:2varsld}
\end{equation}
Our formula (\ref{eq:2varsld}) agrees  with  the formula in \cite[Theorem 3.13, p. 723]{McLean1998}.
\end{example}

\subsection{Second  variation of  the volume  of a coassociative   submanifold}

In this  subsection,   using Theorem \ref{thm:main}, we   give a new  proof of McLean's  second variation formula for  coassociative submanifolds (Theorem \ref{thm:McLeanco}).

Let  $L$ be a coassociative     submanifold  in a $G_2$-manifold $(M^7, \varphi,  g)$.  We  identify  the  normal bundle $NL$  with the bundle $\Lambda_+ ^2 T^*L$ as follows (\cite[Theorem 2.5]{JS2005},  cf \cite[Theorem 4.5]{McLean1998}).  Let us denote by $\Lambda ^2_+ T^*L$ the  bundle of self-dual 2-forms  on $L$. We   define   the  following  map
$$ NL \ni    V \mapsto  \alpha_V: =  (V\rfloor \varphi)_{| L}  \in  \Lambda ^2 _+  T^*  L. $$ 
The following  Lemma is due to McLean.

\begin{lemma}\label{lem:McLeanco} (cf. \cite[Theorem 4.5]{McLean1998},  cf.\cite[Theorem 2.5]{JS2005})   Assume that $L$ is a  coassociative  submanifold in $(M^7, \varphi, g)$ and $V\in \Gamma (NL)$ is a normal vector  field. Then
$${d\over  dt}|_{ t = 0} ((\exp tV)^*\varphi)|_{ L}  = d\alpha _V.$$
\end{lemma}

Now  we are ready to give  a new  proof of the following  Theorem due to McLean.

\begin{theorem}\label{thm:McLeanco} (\cite[Theorem 4.9, p. 731]{McLean1998})  Let $L$ be a    coassociative  submanifold  in a $G_2$-manifold  $(M^7, \varphi, g)$.  For any normal vector field $V \in \Gamma( NL)$  with compact support we have
\begin{equation}
\frac{d^2}{dt^2} |_{t =0} vol  (L_t) = \int _{L} \la  d\alpha_V,  d \alpha _V \ra dvol_x.\label{eq:coavar}
\end{equation}
\end{theorem}

\begin{proof} 
Recall that $\tau$  is defined in  (\ref{eq:tau7}).  As   before we denote  by $\xi_t(x) : = (\exp tV )_*(\xi(x))$,   where $\xi(x)$  is the  decomposable 4-vector associated with $T_xL$. 

\begin{lemma}\label{lem:1var} Let  $L$ be a  compact coassociative     submanifold  and $V\in \Gamma(NL)$.  Then for all $x \in L$ we have
\begin{equation}
\la \nabla_{\p t}( \tau (\xi_t(x)))  , \nabla_{\p t} (  \tau (\xi_t(x)))\ra| _{t=0}= \la  d \alpha_V ,  d \alpha _V    \ra(x).\label{eq:1var}
\end{equation}
\end{lemma}
\begin{proof}
As before we set  $L_t : =  \exp  (tV) (L)$.  
  Define the Poincare duality map  for any  $y \in L_t$
		$$P_t : T_yL_t \to \Lambda  ^3 (T_y^*L_t),\:  v \mapsto   v \rfloor vol_{L_t}(y) .$$
Note that $\tau(\xi_t(x)) \in T_{\exp t V (x)} L_t$. Using $P_t$  and denoting $y: = \exp tV (x)$,   we  rewrite  the relation  (\ref{eq:tau7})       as  follows
\begin{equation}
P_{t} (\tau (\xi_t(y))) = \varphi_{|L_t}\cdot | \xi_t (y)|.\label{eq:poincare}
\end{equation}
 Using  (\ref{eq:poincare}),  noting      that  $\varphi  |_{ L} = 0$  and $| \xi_0 (x)| = 1$,  we obtain 
\begin{equation}
\nabla _{\p t} (  \tau (\xi_t))_{|t =0} (x)  = (P_0)^{-1} (\nabla _{\p t} (\varphi|_{  L_t}) _{ | t =0} (x)). \label{eq:poincare2}
\end{equation}
Denote  by $\Pi_t$ the  parallel transportation   from $\exp t V  (x) $ to $ x$ along   the  curve $\exp t V (x)$  that is induced by   the   connection $\nabla$,    and       abbreviate  
$$\varphi_t (x): = \varphi (\exp t V ) (x) |_{ L_t}, \:   D_t : =  (\exp tV)^*.$$ 
Then we   have
\begin{eqnarray}
\nabla _{\p t}  (\varphi|_{  L_t}) _{ | t =0} (x) = { d\over dt}|_{ t =0} [\Pi_t\circ  D_t^{-1}\circ  D_t ( \varphi_t  (x) ) ]= \nonumber \\
= {d\over dt }|_{ t =0} [D_t(\varphi_t (x))] \label{eq:flat1}
\end{eqnarray}
since $ \varphi_0 (x)= 0$. 
From (\ref{eq:flat1})
we obtain
\begin{equation}
\nabla _{\p t}  (\varphi|_{  L_t}) _{ | t =0} (x) = {d\over dt}|_{ t =0}  [( \exp  t V)  ^* (\varphi)]|_L .\label{eq:flat}
\end{equation}
 Using Lemma \ref{lem:McLeanco}, noting that $P_0$  is an isometry,    we derive  Lemma \ref{lem:1var} from (\ref{eq:poincare2})  and (\ref{eq:flat})  immediately.
\end{proof}

{\it Continuation  of  the  proof of Theorem \ref{thm:McLeanco}}. Clearly  Theorem  \ref{thm:McLeanco} follows from  Theorem \ref{thm:main} and  Lemmas \ref{lem:1var}, \ref{lem:tau7}.
\end{proof}

\begin{remark}\label{rem:cartan}
In \cite[p. 736]{McLean1998} McLean  gave  a  short proof of the following  formula
	\begin{equation}
{d\over dt}_{|t =0 }  ( \exp (tV)^* ( \chi))_{| L} = \Ds (V) \cdot  vol_L \in \Om^3(L, NL). \label{eq:infvarm}
\end{equation}
This formula,  which  looks   like  (\ref{eq:infvar}), was important   for   McLean's computation  of   infinitesimal   deformations of  associative  submanifolds.  Unfortunately,  in his proof McLean  applied  the Cartan formula  $\Ll _V (\phi) =  d(V \rfloor \phi) + V \rfloor  d\phi$    for   scalar  valued   differential forms  $\phi$ to   the tangent bundle  valued  forms  $\varphi$.
Using   the   argument   in   the proof  of Lemma \ref{lem:1var} 
 we  can easily  prove
(\ref{eq:infvarm}).    To  prove (\ref{eq:infvarm})  was one  of our  motivations  to revisit McLean's     second variation formulas.
\end{remark}

\section{Second variation formula for Cayley   submanifolds}\label{sec:spin7}

In this   section  we   give a new proof  of McLean's  second variation   formula  for a  compact Cayley submanifold  in  a Spin(7)-manifold  (Theorem \ref{thm:McLeanC}), correcting   a coefficient in  the RHS  of  Formula (6.16)  in \cite[p. 743]{McLean1998}, which is twice larger than our coefficient.  

\subsection{Cayley submanifolds in Spin(7)-manifolds  and cross  products}  In this subsection we recall  basic facts   concerning  Cayley submanifolds  in Spin(7)-manifolds that are important for
understanding of our proof of  McLean's second variation formula   for Cayley submanifolds.  Our main sources are \cite{HL1982}, \cite{Fernandez1986}, \cite{McLean1998}, \cite{Ohst2014}.

Let $(x_1, \cdots, x_8)$ be coordinates of $\R^8$. 
Define a 4-form $\Phi_0$ on $\R^8$ by \footnote{there are many choices of coordinates  on  $\O$,  which result
in  seemingly different $\Phi_0$ in  different papers  on  Spin(7)-geometry. Here we consistently follow  \cite{Ohst2014}, which agrees with \cite[Corollary 3.1,p. 120]{HL1982}}(\cite[(2.1)]{Ohst2014})
\begin{align*}
\Phi_0 =& dx^{1234} + dx^{1256} - dx^{1278} + dx^{1357} + dx^{1368} + dx^{1458} -dx^{1467}\\
           &- dx^{2358} + dx^{2367} + dx^{2457} + dx^{2468} - dx^{3456} + dx^{3478} + dx^{5678}, 
\end{align*}
where $dx^{i_1 \dots i_4}$ is an abbreviation of $dx^{i_1} \wedge \cdots \wedge dx^{i_4}$. 
The subgroup of ${\rm GL}(8, \R)$ preserving $\Phi_0$ is  Spin(7).  Let $g_0$ denote  the standard  metric on $\R^8$.  We call $(\Phi_0, g_0)$  
{\it the standard  ${\rm Spin(7)}$-structure}.

Let $M^8$ be an oriented 8-manifold and 
$\Phi$ be a 4-form on $M^8$. 
A   4-form $\Phi$ is called a {\it ${\rm Spin}(7)$-structure} on $M^8$ if 
for each $p \in M$, there exists an oriented isomorphism $I_p$ between $T_{p}M^8$ and $\mathbb{R}^{8}$ 
identifying $\Phi_{p}$ with $\Phi_{0}$. 
Then $\Phi$ induces the metric $g_\Phi$ by  pulling back  the   metric $g_0$  using $I_p$.  
Since ${\rm Spin}(7)$ is a subgroup of ${\rm SO}(8)$, 
$g_\Phi$ does not  depend on the choice of $I_p$.
In our paper  we shall consider    only  Spin (7)-manifolds $(M^8, \Phi, g)$, i.e.       manifolds   with $d\Phi =0$.

 The   4-form $\Phi_0$ has been discovered  by Harvey-Lawson in \cite{HL1982}, where they  call  it {\it the Cayley calibration}.
 
On a  Spin(7)-manifold $(M^8, \Phi, g)$     we define  
 a triple   cross  product $P\in \Om^3(M^8, TM^8)$  as follows   
\begin{equation}
\Phi (x, y, z, w) = \la x, P(y, z,w) \ra \label{eq:P}
\end{equation}

We shall show that   the  Cayley calibration $\Phi$ on any  Spin(7)-manifold  $(M^8, \Phi, g)$ satisfies  Harvey-Lawson's    identity.
To define    a bundle $E$ on $M^8$ and $\Psi \in \Om^4 (M^8, E)$   such that  $(\Phi, \Psi)$ satisfy  (\ref{eq:HLI2}),   we need  recall  the notion of the cross  product  on   $M^8$.

First we need the following  (point-wise) splitting  on   $(M^8, \Phi, g)$
$$\Lambda ^2 T^*M^8 = \Lambda ^2 _7 T^*M^8 \oplus  \Lambda ^2 _{21} T^*M^8, $$
where $\Lambda ^2 _kT^*M^8$ corresponds to  an irreducible  Spin(7)-module of dimension $k$ in  the  Spin(7)-module $\Lambda^2T^*M ^8$.

For a tangent vector $v \in TM^8$, define a cotangent vector $v^{\flat} \in T^* M^8$ 
by $v^{\flat} = g(v, \cdot)$. Define a 2-fold cross product $TM^8 \times TM^8 \rightarrow \Lambda^2_7 T^* M^8$ by 
$$
v \times w = 2 \pi_7 (v^{\flat} \wedge w^{\flat})
= \frac{1}{2} \left( v^{\flat} \wedge w^{\flat} - * (v^{\flat} \wedge w^{\flat} \wedge \Phi) \right)
$$
for $v, w \in TM^8$, where $\pi_7$ denotes the  projection  to $\Lambda ^2 _7 T^*M^8$ according  to the  above splitting  of $\Lambda^2 T^*M^8$.

Now we set  $\tau \in \Om ^4 (M^8, \Lambda ^2 _7 T^*M^8)$ as  follows
\cite[(2.7)]{Ohst2014}.
\begin{equation}
\tau(a,b,c,d) : = - a \times P(b, c ,  d) +  \la a, b\ra (c \times d) + \la a, c\ra (d \times b) + \la a, d\ra (b \times c). \label{eq:tau}
\end{equation}


The  following  Lemma  asserts that $\Phi$ satisfies Harvey-Lawson's identity.

\begin{lemma}\label{lem:hlcay}(\cite[Theorem 1.28, p. 119]{HL1982})  For all $x, y, z, w \in TM^8$  we have
$$\Phi (x \wedge y \wedge z \wedge w) ^2 + |\tau(x, y,z,w)|^2 =  | x \wedge y \wedge z \wedge w| ^2 .$$
\end{lemma}  
\begin{remark}\label{rem:tau} (\cite[Proposition IV.B.14, p. 149]{HL1982})   Let $(M^7, \varphi, g)$  be a $G_2$-manifold. Then  $(S^1 \times M^7,  d\theta \wedge \varphi + * \varphi,  d\theta ^2  + g)$ is
 a Spin(7)-manifold. Furthermore,  for any $\theta \in S^1$,  the restriction of    $\tau$   on $M^8$  to  $\{ \theta \} \times M^7$  is    equal to  the   4-form $\tau$  defined  in (\ref{eq:tau7}).
Thus  we use the notation $\tau$ for   the form on $M^7$ as well   as  for  the form on $M^8$.
\end{remark}

	

Recall that  a 4-submanifold $L$ in a Spin(7)-manifold  $(M^8, \Phi, g)$ is called  {\it Cayley}, if $L$ is calibrated by $\Phi$, i.e.  $\Phi |_{ L} = vol_L$.

\begin{example}\label{ex:coass1}  Assume that $(M^7, \varphi, g)$  be a $G_2$-manifold  and  $L$ is a coassociative  submanifold  in  $(M^7, \varphi, g)$.  Then $(S^1 \times M^7, dt\wedge \varphi + * \varphi, dt^2 + g)$  is a Spin(7)-manifold  and $\{\theta\} \times L$ is its Cayley  submanifold  for any $\theta \in  S^1$.
\end{example}

\subsection{The normal bundle of a Cayley submanifold  and   its associated  Dirac type  operator}
We collect   known results   from  \cite[Section 6]{McLean1998}  and \cite[\S 2, 3]{Ohst2014}.

Let  $L$ be a Cayley submanifold  in   a Spin(7)-manifold $(M^8, \Phi, g)$.   Then  the bundle $\Lambda ^2_- T^* L$ of anti-self dual  2-forms on
$L$ is isomorphic  to a subbundle  of   the bundle $\Lambda ^2 _7 T^*M^8|_{L}  \subset \Lambda ^2 T^*M^8|_{L}$ via the   following embedding (\cite[Section 6]{McLean1998}, see also \cite[Section 2]{Ohst2014})
\begin{equation}
\Lambda ^2_-T^*L \to \Lambda ^2 _7 T^*M^ 8|_{L} , \:   \alpha \mapsto 2 \pi_7 (\alpha) = {1\over 2} ( \alpha - * ( \alpha \wedge \Phi)),\label{eq:e1}
\end{equation}
where   we extend  $\alpha  \in \Lambda ^2 _- T^*L$ to $\Lambda ^2 T^*M^8 |_{ L}$ by $v\rfloor \alpha = 0$  for all $ v \in NL$, and
$\pi_7$ is  defined above.  
Let  $E_L$ denote  the orthogonal  complement   of $\Lambda ^2 _- T^*L$  in  $\Lambda ^2 _7 T^*M^8|_{L}$, i.e.
$$ \Lambda ^2 _7 T^*M^8|_{L} \cong  \Lambda_- ^2T^* L \oplus  E_L.$$
Note that      $E_L$ has rank 4. 
Furthermore  the cross  product   restricts  to $ TL \times NL  \to E_L$. 
Now we define  a    Dirac type   operator $D: \Gamma (NL) \to \Gamma (E_L)$ as follows (cf.  Subsection 3.2)
$$ D (s) : =  \sum _{i =1} ^4  e _i \times \nabla _{e_i} ^\perp  s,$$
where $e_i$ denote  a positive orthonormal basis  of $T_xL$  and $\nabla  ^\perp$ is the induced  connection on the normal bundle.

\subsection{Second   variation of  the volume of a  Cayley submanifold}
To     derive   the   second variation  formula  we  use the  following   Lemma, which is an  analogue of  Lemma \ref{lem:gayet}.

\begin{lemma}\label{lem:ohst} (\cite[Theorem 3.1]{Ohst2014}) Let $\xi(x)$  denote the unit  decomposable  4-vector  associated with the tangent  space $T_xL$. Then
\begin{equation}
\nabla_{\p t}  ( \tau( (\exp TV)_*\xi(x)))_{| t =0} =  D (V)(x) \in E_L(x). \label{eq:infvarc}
\end{equation}
\end{lemma}

\begin{theorem}\label{thm:McLeanC} (\cite[Theorem 6.4, p. 743]{McLean1998}) Let $L$ be a  compact  Cayley submanifold in   a Spin(7)-manifold  $(M^8, \Phi, g)$.  For any   normal vector field $V$ on $L$ with compact support, the second  variation of  the volume  of $L$  with  the variation  field $V$  is given by
\begin{equation}
\frac{d^2}{dt^2} |_{t =0} vol  (L_t) = \int _{L} \la  D (V), D (V) \ra dvol_x.\label{eq:cayleyvar}
\end{equation}
\end{theorem}

\begin{proof} Theorem \ref{thm:McLeanC} is    a consequence of Theorem \ref{thm:main} and   Lemmas  \ref{lem:hlcay}, \ref{lem:ohst}.
\end{proof}

\begin{remark}\label{rem:mistake} We  would like to  explain  where  McLean   made  mistakes   leading  to  his  Theorems \cite[Theorem 5.3]{McLean1998} and  \cite[Theorem 6.4]{McLean1998} concerning the  second variation formulas for associative  and Cayley submanifolds.  His    general   formula \cite[Theorem 2.4, p. 711]{McLean1998} for second  variation  of  calibrated submanifolds  seems to be correct, at least we do not find  any mistake  in  McLean's application of that formula  to the  special Lagrangian  and coassociative submanifolds.   His computation    before  the end of the proof of Theorem 5.3 in \cite[p. 737]{McLean1998}  also agrees  with our  formula, but McLean suddenly added  a coefficient 2  to  his formula  in \cite [Theorem 5.3]{McLean1998}, referring  to \cite[Theorem 2.4]{McLean1998}, which has been  misprinted  there as Theorem 1.4.
The same  mistake  has been repeated  in  McLean's proof  of \cite[Theorem 6.4]{McLean1998}  in \cite[p. 743]{McLean1998}:  the last formula in McLean's proof of  \cite[Theorem 6.4]{McLean1998}  agrees with ours, but  McLean   added a coefficient 2  to his formula  \cite[Theorem 6.4]{McLean1998}, referring  to an  irrelevant  formula (2.13) in his paper.     We guess  that  McLean  worked  on    several versions of his paper and did not check  all formulas carefully, see  also   McLean's citation of Simon's  formulas   in \cite[p. 707, 717]{McLean1998}, which   are not consistent.
\end{remark}

\subsection*{Acknowledgement}  We thank  Kotaro Kawai  for a  helpful comment  on an early version of this paper  and  an anonymous  referee  for his  suggestions which improve    the exposition of our note.

\end{document}